\documentclass[reqno]{amsart}
\usepackage{amssymb,amsmath}%
\usepackage{ifpdf}
\ifpdf
\else
\expandafter \ifx \csname dvipdfm \endcsname \relax
\usepackage[hypertex,hyperindex]{hyperref}
\else
\usepackage[dvipdfm,hyperindex]{hyperref}
\fi
\fi

\begin{document}
\title[ Coincidence point results]
{Coincidence point results involving a generalized class of simulation functions}

\author[D.K. Patel, P.R. Patle, L. Budhia, D. Gopal]
{D.K. Patel, P.R. Patle, L. Budhia, D. Gopal}  

\address{D.K. Patel \newline
 Department of Mathematics, Visvesvaraya National Institute of Technology, Nagpur-440010, India}
\email{deepesh456@gmail.com}

\address{P.R. Patle \newline
 Department of Mathematics, Visvesvaraya National Institute of Technology, Nagpur-440010, India}
\email{pradip.patle12@gmail.com}

\address{L. Budhia \newline
Department of Applied Mathematics \& Humanities, S.V. National Institute of Technology, Surat-395007, India}
\email{lokesh86budhia@gmail.com}

\address{D. Gopal \newline
Department of Applied Mathematics \& Humanities, S.V. National Institute of Technology, Surat-395007, India}
\email{gopaldhananjay@yahoo.in}

\subjclass[2000]{47H10, 54H25}
\keywords{Coincidence Point; $\mathcal{Z}$-contraction; Suzuki type $(Z_G,g)$-contractions}

\begin{abstract}
The purpose of this work is to introduce a general class of $C_G$-simulation functions and obtained some new coincidence and common fixed points results in metric spaces. Some useful examples are presented to illustrate our theorems. Results obtained in this paper extend, generalize and unify some well known fixed and common fixed point results. 
\end{abstract}

\maketitle
\numberwithin{equation}{section}
\newtheorem{theorem}{Theorem}[section]
\newtheorem{definition}{Definition}[section]
\newtheorem{lemma}[theorem]{Lemma}
\newtheorem{proposition}[theorem]{Proposition}
\newtheorem{corollary}[theorem]{Corollary}
\newtheorem{remark}{Remark}[section]
\newtheorem{example}{Example}[section]

\section{Introduction}
Because of its large number of applications in various branches  of mathematical sciences, fixed point theory enticed many mathematicians. The most relevant result of the theory is Banach contraction  principle (BCP).  Several authors have studied on various generalizations of Banach contraction mapping principle. For important and famous examples of such generalizations, we may cite Edelstein \cite{Edelstein}, Kannan \cite{Kannan}, Reich \cite{Reich71, Reich71a,Reich72}, Ciric \cite{Ciric} and Suzuki \cite{Suzuki2008}. Recently, BCP has been generalized in a new way by Khojasteh et al. in \cite{KhojastehShuklaRadenovic} defining $\mathcal{Z}$-contraction with the help of a class of control functions called simulation functions. In \cite{OlgunBicer2016}, Olgun et al. obtained the fixed point results for generalized $\mathcal{Z}$-contractions. Further, L-de-Hierro et al. \cite{HierroKarapinar} enlarged this class of simulation functions and extended some coincidence and common fixed point theorems.

All the above results, require the contraction condition or the contraction type condition to hold for every possible pair of points in domain. In order to relax this requirement in accordance with the outcome of the result Suzuki
\cite{Suzuki2008} contributed significantly. The results of Suzuki has inspired many researcher working in the area of metric fixed point theory (see \cite{PantSinghMishra2016},\cite{Suzuki2009}  and references therein).
Recently, Kumam et al. in \cite{KumamGopalBudhia} defined Suzuki type $\mathcal{Z}$-contraction and unified some known fixed point results. In \cite{RadenovicChandok}, Radenovic and Chandok coined the notion of $(Z_G,g)$-contraction and generalized the results of Jungck \cite{Jungck1976}, L-de-Hierro et al. \cite{HierroKarapinar} and Olgun et al. \cite{OlgunBicer2016}. Motivated by the results of \cite{KumamGopalBudhia} and \cite{RadenovicChandok}, in this work we introduce the intuition of Suzuki type $(Z_G,g)$-contractions and obtain some coincidence point theorems. 
An application to fractional order functional differential equation is presented to illustrate the usability of the obtained result.

\section{Preliminaries}
\label{sec:1}
We now revisit some concepts and results in the literature.
\begin{definition}\rm\cite{Ansari2014,RadenovicChandok} \label{Cclassfunction}
A mapping $G : [0,\infty)^2 \to \mathbb{R}$ is called $C$-class function if it is
continuous and satisfies the following conditions:
\begin{enumerate}
		\item[(i)]  $G(s, t) \leq s$,
		\item[(ii)] $G(s, t) = s$ implies that either $s = 0$ or $t = 0$, for all $s, t \in [0,\infty)$.
\end{enumerate}
\end{definition}

\begin{definition}\rm\cite{RadenovicChandok}\label{CGproperty}
	A mapping $G : [0,\infty)^2 \to \mathbb{R}$ has a property $C_G$, if there exists an
	$C_G \ge 0$ such that
	\begin{enumerate}
		\item[(i)]  $G(s, t) > C_G$ implies $s > t$,
		\item[(ii)] $G(t, t) \le C_G$ for all $t \in [0,\infty)$.
	\end{enumerate}
\end{definition}
\begin{definition}\rm\cite{ArgoubiSametVetro2015}\label{CGSimulationfunction}
	A simulation function is a mapping $\zeta : [0,\infty)^2 \to \mathbb{R}$ satisfying the following:
	\begin{enumerate}
		\item[(i)] $\zeta (t, s) < s-t$ for all $t, s > 0$,
		\item[(ii)] if $\{t_n\}$ and $\{s_n\}$ are sequences in $(0,\infty)$ such that $\lim\limits_{n\to\infty} t_n = \lim\limits_{n\to\infty} s_n > 0$ and
		$t_n < s_n$, then $\lim\limits_{n\to\infty}\sup \zeta(t_n, s_n) < 0.$
	\end{enumerate}
\end{definition}                                                                                               
For examples and related results on simulation functions, one may refer to  \cite{ChenTang2016, HierroKarapinar, HieroSamet, Karapinar2016, KhojastehShuklaRadenovic, KarapinarKhojasteh, KomalKumamGopal2016, NastasiVetro2015, RadenovicVetroVujakovic2017, Samet2015, TchierVetro2016}.
\label{AAAAAAA}
\begin{definition}\rm\cite{AnsariIsikRadenovic,RadenovicChandok}\label{CGSimulationfunction}
	A $C_G$-simulation function is a mapping $\zeta : [0,\infty)^2 \to \mathbb{R}$ satisfying the following:
	\begin{enumerate}
		\item[(a)] $\zeta (t, s) < G(s, t)$ for all $t, s > 0$, where $G : [0,\infty)^2 \to \mathbb{R}$  is a $C$-class function;
		\item[(b)] if $\{t_n\}$ and $\{s_n\}$ are sequences in $(0,\infty)$ such that $\lim\limits_{n\to\infty} t_n = \lim\limits_{n\to\infty} s_n > 0$ and
		$t_n < s_n$, then $\lim\limits_{n\to\infty}\sup \zeta(t_n, s_n) < C_G.$
	\end{enumerate}
\end{definition}                                                                                               
Let $Z_G$ be the family of all $C_G$-simulation functions $\zeta : [0,\infty)^2 \to \mathbb{R}.$
\begin{definition}\rm \cite{RadenovicChandok}\label{ZG-contraction}
	Let $(X, d)$ be a metric space and $f, g : X \to X$ be self-mappings.
	The mapping $f$ is called a $(Z_G, g)$-contraction if there exists $\zeta\in Z_G$ such that
	\begin{equation} \label{RadnovicContrA}
	\zeta(d(fx,fy), d(gx,gy))\geq C_G
	\end{equation}
	for all $x, y \in X$ with $gx \ne gy.$
\end{definition}
If $g = i_X$ (identity mapping on $X$) and $C_G = 0$, we get $\mathcal{Z}$-contraction of \cite{KhojastehShuklaRadenovic}.

\begin{definition}\rm \cite{RadenovicChandok}\label{GeneralizedZG-contraction}
	Let $(X, d)$ be a metric space and $f, g : X \to X$ be self-mappings.
	The mapping $f$ is called a generalized $(Z_G, g)$-contraction if there exists $\zeta\in Z_G$ such that
	\begin{equation} \label{GeneralizedZG-contraction}
	\zeta\left(d(fx,fy), \max\left\lbrace d(gx,gy), d(gx,fx), d(gy,fy), \frac{d(gx,fy)+d(gy,fx)}{2} \right\rbrace\right)\geq C_G
	\end{equation}
	for all $x, y \in X$ with $gx \ne gy.$
\end{definition}
If $g = i_X$ (identity mapping on $X$) and $C_G = 0$, we get $\mathcal{Z}$-contraction of \cite{OlgunBicer2016}.

\begin{definition}\rm \cite{Jungck1996}\label{coincidencepoint}
	Let $f$ and $g$ be self-mappings of a set $X$. If $u = fv = gv$ for some $v \in X$, then
	$v$ is called a coincidence point of $f$ and $g$ and $u$ is called a point of coincidence of $f$ and
	$g$. The pair $(f, g)$ is weakly compatible if the mappings commute at their coincidence
	points.
\end{definition}

\begin{proposition}\rm \cite{AbbasJungck2008} \label{PropA}
	Let $f$ and $g$ be weakly compatible self-mappings of a set $X$. If $f$ and $g$ have a unique point of coincidence $u=fx=gx$, then $u$ is the unique common fixed point of $f$ and $g$.
\end{proposition}

\begin{lemma} \rm \cite{RadenovicChandok} \label{LemmaA}
	Let $(X, d)$ be a metric space and $\{x_n\}$ be a sequence in $X$
	such that $\lim\limits_{n\to \infty}d(x_n,x_{n+1})=0.$ If $\{x_n\}$ is not Cauchy then there exists $\varepsilon>0$ and two subsequences $\{x_{m(k)}\}$ and $\{x_{n(k)}\}$ of $\{x_n\}$ where $n(k)>m(k)>k$ such that
	\begin{eqnarray*}
		&&\lim\limits_{k\to\infty} d(x_{m(k)},x_{n(k)})=\lim\limits_{k\to\infty} d(x_{m(k)},x_{n(k)+1})=\lim\limits_{k\to\infty} d(x_{m(k)-1},x_{n(k)})=\varepsilon\\
		&& ~~~~\text{and}~~~~\lim\limits_{k\to\infty} d(x_{m(k)-1},x_{n(k)+1})=\lim\limits_{k\to\infty} d(x_{m(k)+1},x_{n(k)+1})=\varepsilon. 
	\end{eqnarray*}
\end{lemma}

\section{Main results}
\label{sec:2}

In this section, we introduce the new concept of Suzuki type contractions involving $C_G$-simulation functions which is used to prove several coincidence point theorems in complete metric spaces. 

\begin{definition}\rm
	Let $(X, d)$ be a metric space and $f, g : X \to X$ be self-mappings.
	A mapping $f$ is called a Suzuki type $(Z_G, g)$-contraction if there exists $\zeta\in Z_G$ such that
	\begin{equation}\label{SuzukiContraction}
	\frac{1}{2}d(gx,fx)<d(gx,gy) \implies \zeta(d (fx, fy), d (gx, gy))\ge C_G
	\end{equation}
	for all $x, y \in X$ with $gx \ne gy.$
\end{definition}
If $g = i_X$ (identity mapping on $X$) and $C_G = 0$, we get Suzuki type $\mathcal{Z}$-contraction of \cite{KumamGopalBudhia}.

\begin{definition}\rm \label{SuzukiGeneralizedZG-contraction}
	Let $(X, d)$ be a metric space and $f, g : X \to X$ be self-mappings.
	The mapping $f$ is called a Suzuki type generalized $(Z_G, g)$-contraction if there exists $\zeta\in Z_G$ such that
	\begin{equation}\label{SuzukiGeneralizedZG-contraction}
	\frac{1}{2}d(gx,fx)<d(gx,gy) ~~~\Longrightarrow ~~~\zeta\left(d(fx,fy), M(x,y)\right)\geq C_G
	\end{equation}
	for all $x, y \in X$ with $gx \ne gy$, where
	$$M(x,y)=\max\left\lbrace d(gx,gy), d(gx,fx), d(gy,fy), \frac{d(gx,fy)+d(gy,fx)}{2} \right\rbrace.$$
\end{definition}
If $g = i_X$ (identity mapping on $X$) and $C_G = 0$, we get Suzuki type generalized $\mathcal{Z}$-contraction.

\begin{theorem}\rm \label{TheoremC1}
	Let $(X, d)$ be a metric space, $f, g : X \to X$ be self-mappings and $f$ be a Suzuki type generalized $(Z_G, g)$-contraction. 
	Assume that the following conditions hold:
	\begin{enumerate}
		\item[(i)] $f(X)\subseteq g(X)$,
		\item[(ii)] $g(X)$ or $f(X)$ is complete.
	\end{enumerate}
	Then $f$ and $g$ have unique point of coincidence.
\end{theorem}
\begin{proof}
	Suppose that the point of coincidence exists, then the uniqueness follows as: let $u_1$ and $u_2$ be two distinct point of coincidence of $f$ and $g$ i.e. $u_1=fv_1=gv_1$  and $u_2=fv_2=gv_2$ for some  $v_1,v_2 \in X$. Now, since
	\begin{equation*}
	\frac{1}{2}d(gv_1,fv_1)=0<d(gv_1,gv_2),
	\end{equation*}
	so by using (\ref{SuzukiContraction}), we get 
	\begin{align*}
	C_G & \leq \zeta(d (fv_1, fv_2), M(v_1, v_2)) \\
	&=\zeta(d(u_1,u_2),d(u_1,u_2))<G(d(u_1,u_2),d(u_1,u_2))\le C_G,
	\end{align*}
	which is a contradiction. Thus, the point of coincidence is unique.
	
	\noindent Consider the sequences $\{x_n\}$ and $\{y_n\}$ defined as
	$$y_n=fx_n=gx_{n+1}~\text{for all}~n \in \mathbb{N}\cup\{0\},$$
	and assume that $y_n \ne y_{n+1}$ for all $n \in \mathbb{N}\cup\{0\}.$
	Let $x = x_{n+1}, y = x_{n+2}$, then 
	\begin{align*}
	\frac{1}{2}d(gx_{n+1},fx_{n+1})=\frac{1}{2}d(gx_{n+1},gx_{n+2})<d(gx_{n+1},gx_{n+2}).
	\end{align*}
	So from \eqref{SuzukiGeneralizedZG-contraction}, we have
	\begin{align*}
	C_G &\leq \zeta \Big(d(fx_{n+1}, fx_{n+2}), M(x_{n+1}, x_{n+2})\Big) \\ 
	&= \zeta \Big(d(y_{n+1}, y_{n+2}), \max\left\lbrace d(y_n, y_{n+1}), d(y_{n+1},y_{n+2})\right\rbrace\Big)\\ 
	&< G\Big(\max\left\lbrace d(y_n, y_{n+1}), d(y_{n+1},y_{n+2})\right\rbrace, d(y_{n+1}, y_{n+2})\Big).
	\end{align*}
	Further, using (i) of Definition \ref{CGproperty}, we have $$\max\left\lbrace d(y_n, y_{n+1}), d(y_{n+1},y_{n+2})\right\rbrace > d(y_{n+1}, y_{n+2}).$$ Hence for all for all $n \in \mathbb{N}\cup\{0\}$,  we have  $d (y_n, y_{n+1}) > d (y_{n+1}, y_{n+2})$. So $\{d (y_n, y_{n+1})\}$ is a monotonically decreasing sequence of nonnegative real numbers, and hence there exists $l\geq 0$ such that $\lim\limits_{n\to \infty}d (y_n, y_{n+1})=l$. Assume that $l>0$. Since $d (y_{n+1}, y_{n+2})< M(x_{n+1}, x_{n+2})$, both $d (y_{n+1}, y_{n+2})$ and $M(x_{n+1}, x_{n+2})$ tends to $l$ as $n\to\infty$ and
	\begin{align*}
	\frac{1}{2}d (y_n, y_{n+1})=\frac{1}{2}d(gx_{n+1},fx_{n+1})<d (gx_{n+1}, gx_{n+2})=d (y_n, y_{n+1}),
	\end{align*}
	then using \eqref{SuzukiGeneralizedZG-contraction} and (b) of Definition \ref{CGSimulationfunction}, we get 
	$$C_G \leq \lim\limits_{n\to\infty}\sup \zeta\Big(d (y_{n+1}, y_{n+2}), M(x_{n+1}, x_{n+2})\Big) < C_G,$$ which is a contradiction and hence $l=0$. 
	Now if we assume that $y_n = y_m$ for some $n > m$. Then we can choose $x_{n+1} = x_{m+1}$ and hence also $y_{n+1} = y_{m+1}$. Then using the similar arguments as above we get
	$$d(y_n, y_{n+1})<d (y_{n-1}, y_{n})<\cdots <d (y_m, y_{m+1})=d (y_n, y_{n+1}),$$
	which is a contradiction. Similarly if $y_n = y_m$ for some $n < m$, we get a contradiction. Hence $y_n \neq y_m$ for all  $n \ne m$. 
	
	Now we prove that $\{y_n\}$ is a Cauchy sequence. If not, then by Lemma \ref{LemmaA}  we have 
	$$\lim\limits_{k\to\infty} d(x_{m(k)},x_{n(k)})=\lim\limits_{k\to\infty} d(x_{m(k)+1},x_{n(k)+1})=\varepsilon,$$ 
	and consequently, 
	$$\lim\limits_{k\to\infty}M(x_{m(k)+1}, x_{n(k)+1})=\varepsilon.$$
	We claim that there exists $k_0\in \mathbb{N}$ such that 
	$$\frac{1}{2}d(y_{m(k)},y_{m(k)+1})< d(y_{m(k)},y_{n(k)}),$$ 
	$$\text{i.e.}~~~\frac{1}{2}d(gx_{m(k)+1}, fx_{m(k)+1})<d(gx_{m(k)+1}, gx_{n(k)+1})~~~~~~$$
	for all $k\ge k_0$. If not, then letting $k\to\infty$, we get $\varepsilon \leq 0$, which is not true. Also, using \eqref{SuzukiGeneralizedZG-contraction} and (i) of Definition \ref{CGproperty}, we have $d (y_{m(k)+1}, y_{n(k)+1})<M(x_{m(k)+1}, x_{n(k)+1})$, then  by (b) of Definition \ref{CGSimulationfunction}, we get 
	$$C_G \leq \lim\limits_{k\to\infty}\sup\zeta \Big(d (y_{m(k)+1}, y_{n(k)+1}), M(x_{m(k)+1}, x_{n(k)+1})\Big) < C_G$$ a contradiction. Hence $\{y_n\}$ is a Cauchy sequence.
	
	Suppose  $g(X)$ is complete subspace of $X$. Then $\{y_n\}$ being contained in $g(X)$ has a limit in $g(X)$, say $z\in X$ such that $y_n \to gz$, i.e. $gx_n \to gz$ as $n\to \infty$. We will show that $z$ is  the coincidence point of $f$ and $g$. We  can suppose $y_n \ne fz,  gz$ for  $n\in\mathbb{N}\cup\{0\}$, otherwise we are done. 
	We claim that 
	\begin{equation}\label{EQC1}
	\frac{1}{2}d(gx_n,gx_{n+1})<d(gx_n,gz)~~\text{or}~~\frac{1}{2}d(gx_{n+1},gx_{n+2})<d(gx_{n+1},gz)
	\end{equation}
	for every $n\in\mathbb{N}\cup\{0\}$. If not, then there exists $m\in \mathbb{N}$ for which 
	$$\frac{1}{2}d(gx_m,gx_{m+1})\ge d(gx_m,gz)~~\text{and}~~\frac{1}{2}d(gx_{m+1},gx_{m+2})\ge d(gx_{m+1},gz)$$ holds. Then
	$$2d(gx_m,gz)\le d(gx_m,gx_{m+1}) \le d(gx_m,gz)+d(gz,gx_{m+1})$$ which implies $d(gx_m,gz)\le d(gz,gx_{m+1})\leq \frac{1}{2}d(gx_{m+1},gx_{m+2}).$ Further since $$\frac{1}{2}d(gx_m,fx_{m})=\frac{1}{2}d(gx_m,gx_{m+1}) <d(gx_m,gx_{m+1}),$$ so from \eqref{SuzukiContraction}, we have 
	$$C_G \le \zeta(d(fx_m,fx_{m+1}),M(x_m,x_{m+1}))< G(M(x_m,x_{m+1}),d(fx_m,fx_{m+1})),$$
	and in view of (i) of Definition \ref{CGproperty}, we get $M(x_m,x_{m+1})=d(gx_m,gx_{m+1})>d(fx_m,fx_{m+1})$. Now 
	\begin{align*}
	d(fx_m,fx_{m+1})&<d(gx_m,gx_{m+1}) \\
	&\leq d(gx_m,gz)+ d(gz,gx_{m+1})\\
	&\leq \frac{1}{2}d(gx_{m+1},gx_{m+2})+\frac{1}{2}d(gx_{m+1},gx_{m+2})\\
	&= d(gx_{m+1},gx_{m+2}) = d(fx_m,fx_{m+1}),
	\end{align*}
	which is a contradiction. Hence \eqref{EQC1} holds. Now from \eqref{EQC1} and \eqref{SuzukiContraction} we have 
	%
	$$C_G \leq \zeta\Big(d(fx_n, fz), M(x_n, z)\Big) < G\Big(M(x_n, z), d(fx_n, fz)\Big)$$ and then by (i) of Definition \ref{CGproperty}, we get $d(fx_n, fz)<M(x_n, z).$ Now, if we assume $fz\neq gz$, then taking limit as $n\to \infty$, we get $d(gz,fz)<d(gz,fz)$, a contradiction. Hence $gz=fz.$ In same manner if $\frac{1}{2}d(gx_{n+1},gx_{n+2})<d(gx_{n+1},gv)$ we can show that $gz=fz$. In case $f(X)$ is a complete subspace of $X$, the sequence $\{y_n\}$ converges in $g(X)$ since $f(X) \subseteq g(X)$. So the previous argument works. 
\end{proof}
\begin{theorem}\rm \label{TheoremC2}
	Let $(X, d)$ be a complete metric space, $f, g : X \to X$ be self-mappings and $f$ be a Suzuki type generalized $(Z_G, g)$-contraction. 
	Assume that the following conditions hold:
	\begin{enumerate}
		\item[(i)] $f(X)\subseteq g(X)$,
		\item[(ii)] $g$ is a continuous,
		\item[(iii)] $f$ and $g$ are commuting.
	\end{enumerate}
	Then $f$ and $g$ have unique point of coincidence.
\end{theorem}
\begin{proof} If the coincidence point exist then it is unique (as proved in Theorem \ref{TheoremC1}).
	Construct the sequences $\{x_n\}$ and $\{y_n\}$ as
	$$y_n=fx_n=gx_{n+1}~\text{for all}~n \in \mathbb{N}\cup\{0\}.$$ Then the sequence $\{y_n\}$ is Cauchy (as shown in proof of Theorem \ref{TheoremC1}).
	
	Since the space $(X,d)$ is complete, there exists $z\in X$ such that $y_n \to z$ or $gx_n\to z$ as $n \to \infty$.  As $g$ is continuous, $g^2x_n \to gz$ as $n \to \infty$. We claim that 
	\begin{equation}\label{EQC2}
	\frac{1}{2}d(ggx_n,ggx_{n+1})<d(ggx_n,ggz)~~\text{or}~~\frac{1}{2}d(ggx_{n+1},ggx_{n+2})<d(ggx_{n+1},ggz)
	\end{equation}
	for every $n\in\mathbb{N}\cup\{0\}$. If not, then there exists $m\in \mathbb{N}$ for which $$\frac{1}{2}d(ggx_m,ggx_{m+1})\ge d(ggx_m,ggz)~~\text{and}~~\frac{1}{2}d(ggx_{m+1},ggx_{m+2})\ge d(ggx_{m+1},ggz)$$ holds. 
	Then
	$$2d(ggx_m,ggz)\le d(ggx_m,ggx_{m+1}) \le d(ggx_m,ggz)+d(ggz,ggx_{m+1})$$ which implies $$d(ggx_m,ggz)\le d(ggz,ggx_{m+1})\leq \frac{1}{2}d(ggx_{m+1},ggx_{m+2}).$$ Further since $f$ and $g$ are commutative, we have
	$$\frac{1}{2}d(ggx_m,fgx_{m})=\frac{1}{2}d(ggx_m,gfx_{m}) <d(ggx_m,ggx_{m+1}).$$
	So from \eqref{SuzukiGeneralizedZG-contraction}, we have 
	$$C_G \le \zeta\Big(d(fgx_m,fgx_{m+1}),M(gx_m,gx_{m+1})\Big)< G\Big(M(gx_m,gx_{m+1}),d(fgx_m,fgx_{m+1})\Big)$$
	and in view of (i) of Definition \ref{CGproperty}, we get $M(gx_m,gx_{m+1})=d(ggx_m,ggx_{m+1})>d(fgx_m,fgx_{m+1})$. Now
	\begin{align*}
	d(fgx_m,fgx_{m+1})& <d(ggx_m,ggx_{m+1}) \\
	&\leq d(ggx_m,ggz)+ d(ggz,ggx_{m+1})\\
	&\leq  \frac{1}{2}d(ggx_{m+1},ggx_{m+2})+\frac{1}{2}d(ggx_{m+1},ggx_{m+2})\\
	&= d(ggx_{m+1},ggx_{m+2})=d(fgx_m,fgx_{m+1}),
	\end{align*}
	which is a contradiction. Hence \eqref{EQC2} holds. Now from \eqref{EQC2} and \eqref{SuzukiGeneralizedZG-contraction}, we have
	$$C_G \le  \zeta\Big(d(fgx_n, fz), M(gx_n, z)\Big) < G\Big(M(gx_n, z), d (fgx_n, fz)\Big)$$ and then by (i) of Definition \ref{CGproperty} and $f,g$ are commuting, we get $d(fgx_n, fz)=d(gfx_n, fz)<M(gx_n, z)$. Now, if we assume $fz\neq gz$, then taking limit as $n\to \infty$ and by continuity of $g$, we get $d(gz,fz)< d(gz,fz)$, a contradiction. Hence $gz=fz.$ In same manner  if $\frac{1}{2}d(ggx_{n+1},ggx_{n+2})<d(ggx_{n+1},ggv)$ we can show that $gz=fz.$ 
\end{proof}

\begin{theorem}\rm \label{TheoremA1}
	Let $(X, d)$ be a metric space, $f, g : X \to X$ be self-mappings and $f$ be a Suzuki type $(Z_G, g)$-contraction.  
	Assume that the following conditions hold:
	\begin{enumerate}
		\item[(i)] $f(X)\subseteq g(X)$,
		\item[(ii)] $g(X)$ or $f(X)$ is complete.
	\end{enumerate}
	Then $f$ and $g$ have unique point of coincidence.
\end{theorem}
\begin{proof} 
	If the coincidence point exist then it is unique (as proved in Theorem \ref{TheoremC1}).
	Construct the sequences $\{x_n\}$ and $\{y_n\}$ as
	$$y_n=fx_n=gx_{n+1}~\text{for all}~n \in \mathbb{N}\cup\{0\}.$$ Then the sequence $\{y_n\}$ is Cauchy (as shown in proof of Theorem \ref{TheoremC1}).
	
	Suppose $g(X)$ is complete subspace of $X$. Then $\{y_n\}$ being contained in $g(X)$ has a limit in $g(X)$, say $z\in X$ such that $y_n \to gz$, i.e. $gx_n \to gz$ as $n\to \infty$. We will show that $z$ is  the coincidence point of $f$ and $g$. We  can suppose $y_n \ne fz$ and $y_n \ne  gz$ for  $n\in\mathbb{N}\cup\{0\}$, otherwise we are done.  We also claim that 
	\begin{equation} \label{EQA1}
	\frac{1}{2}d(gx_n,gx_{n+1})<d(gx_n,gz)~~\text{or}~~\frac{1}{2}d(gx_{n+1},gx_{n+2})<d(gx_{n+1},gz)
	\end{equation}
	for every $n\in\mathbb{N}\cup\{0\}$. If not, then there exists $m\in \mathbb{N}$ for which 
	$$\frac{1}{2}d(gx_m,gx_{m+1})\ge d(gx_m,gz)~~\text{and}~~\frac{1}{2}d(gx_{m+1},gx_{m+2})\ge d(gx_{m+1},gz)$$ holds. Then
	$$2d(gx_m,gz)\le d(gx_m,gx_{m+1}) \le d(gx_m,gz)+d(gz,gx_{m+1})$$ which implies $d(gx_m,gz)\le d(gz,gx_{m+1})\leq \frac{1}{2}d(gx_{m+1},gx_{m+2}).$ Further since $$\frac{1}{2}d(gx_m,fx_{m})=\frac{1}{2}d(gx_m,gx_{m+1}) <d(gx_m,gx_{m+1}),$$ so from \eqref{SuzukiContraction}, we have 
	$$C_G \le \zeta(d(fx_m,fx_{m+1}),d(gx_m,gx_{m+1}))< G(d(gx_m,gx_{m+1}),d(fx_m,fx_{m+1})),$$
	and in view of (i) of Definition  \ref{CGproperty}, we get $d(gx_m,gx_{m+1})>d(fx_m,fx_{m+1}).$
	\begin{align*}
	\text{Now}~~~~~d(fx_m,fx_{m+1})&<d(gx_m,gx_{m+1}) \\
	&\leq d(gx_m,gz)+ d(gz,gx_{m+1})\\
	&\leq \frac{1}{2}d(gx_{m+1},gx_{m+2})+\frac{1}{2}d(gx_{m+1},gx_{m+2})\\
	&= d(gx_{m+1},gx_{m+2}) = d(fx_m,fx_{m+1})
	\end{align*}
	which is a contracdiction. Hence \eqref{EQA1} holds. Now from \eqref{EQA1} and \eqref{SuzukiContraction}, we have  
	$$C_G \leq  \zeta\Big(d(fx_n, fz), d(gx_n, gz)\Big) < G\Big(d(gx_n, gz), d(fx_n, fz)\Big)$$ and then by (i) of Definition \ref{CGproperty}, we get $ d (fx_n, fz)<d (gx_n, gz).$ Taking limit as $n\to \infty$ we get $fx_n \to fz.$ Hence $gz=fz.$ In same manner  if $\frac{1}{2}d(gx_{n+1},gx_{n+2})<d(gx_{n+1},gv)$, we can show that $gz=fz$. In case $f(X)$ is a complete subspace of $X$, the sequence $\{y_n\}$ converges in $g(X)$ since $f(X) \subseteq g(X)$. So the previous argument works. 
\end{proof}

\begin{theorem}\rm \label{TheoremA2}
	Let $(X, d)$ be a complete metric space, $f, g : X \to X$ be self-mappings and $f$ be a Suzuki type $(Z_G, g)$-contraction. Assume that the following conditions hold:
	\begin{enumerate}
		\item[(i)] $f(X)\subseteq g(X)$,
		\item[(ii)] $g$ is a continuous,
		\item[(iii)] $f$ and $g$ are commuting.
	\end{enumerate}
	Then $f$ and $g$ have unique point of coincidence.
\end{theorem}
\begin{proof}
	The proof follows in the same manner as Theorem \ref{TheoremC2} and \ref{TheoremA1}. 
\end{proof}
\begin{theorem}\rm \label{TheoremB}
	In addition to hypotheses of Theorem \ref{TheoremC1}-\ref{TheoremA2}, if $f$ and $g$ are weakly compatible,
	then they have a unique common fixed point in $X$.
\end{theorem}

\begin{proof}
	The proof follows from Proposition \ref{PropA}. 
\end{proof}
By taking $g=i_X$ (identity mapping on $X$) and $C_G=0$ in Theorem \ref{TheoremB}, we get the following result of Kumam et al. \cite{KumamGopalBudhia} as a corollary.
\begin{corollary}\rm \label{corollaryA1A2B}
	Let $(X,d)$ be a complete metric space and $f:X \to X$ be a suzuki type $\mathcal{Z}$-contraction with respect to $\zeta$. Then $f$ has a unique fixed point $x^*$ in $X$. 
\end{corollary}

\begin{remark}\rm\label{RemarkA}
	We observed that the Property (K) assumed in the main result (Theorem 2.1) of Kumam et al. \cite{KumamGopalBudhia} is not required if the mapping $f$ is asymptotically regular at every $x\in X$.
\end{remark}

\begin{example}\rm
	Let $X=[0,1]$ and $d$ be the usual metric on $X$. Define $f,g : X \to X$ as follows:
	$$fx = \left\{\begin{array}{lll}
	0 ~& \mbox{if} ~x \in [0, \frac{1}{4}], \\
	&\\
	\frac{1}{20} ~& \mbox{if} ~x \in \left(\frac{1}{4}, \frac{1}{2}\right), \\
	&\\
	\frac{4}{25} ~& \mbox{if} ~ x\in \left[\frac{1}{2}, 1\right]
	\end{array}\right. ~\text{and}~ \hspace{.5cm}  gx = \left\{\begin{array}{lll}
	\frac{x}{3} ~& \mbox{if} ~x \in [0, \frac{1}{4}], \\
	&\\
	\frac{1}{4} ~& \mbox{if} ~ x \in \left(\frac{1}{4}, \frac{1}{2}\right), \\
	&\\
	\frac{1}{3} ~& \mbox{if} ~x\in \left[\frac{1}{2}, 1\right]. \\
	\end{array}\right.$$  
	Consider the functions $\zeta(t,s)=\frac{4}{5}s-t$, $G(s,t)=s-t$ and $C_G=0$, then for $x\in\left(\frac{1}{4},\frac{1}{2}\right)$ and $y\in\left[\frac{1}{2},1\right]$, the mapping $f$ and $g$ do not satisfy the condition \eqref{RadnovicContrA} and hence Theorem 2.1 of Radenovic and Chandok \cite{RadenovicChandok} is not applicable.
	However, $f$ is a Suzuki type $(Z_G, g)$-contraction as well as Suzuki type generalized $(Z_G, g)$-contraction and satisfy all the assumption of Theorem \ref{TheoremC1} and \ref{TheoremA1}.  
	Here $f$ and $g$ have a unique coincidence point $x=0$. Also, the mappings are weakly compatible.
\end{example}

\begin{example}\rm
	Let $X=\{(0,0), (3,3), (4,0), (0,4), (4,5), (5,4)\}$ be a metric space endowed with the metric $d$ defined by
	$$d((x_1,x_2), (y_1,y_2))= \vert x_1-y_1\vert +\vert x_2-y_2\vert.$$
	Let $f,g:X\to X$ defined by
	$$f(x_1,x_2) = \left\{\begin{array}{lll}
	(0,0) & \mbox{if} ~x_1=x_2, \\[3pt]
	(x_1,0) & \mbox{if} ~x_1< x_2, \\[3pt]
	(0,x_2) & \mbox{if} ~x_1>x_2,
	\end{array}\right.$$
	$$
	g(x_1,x_2) = \left\{\begin{array}{lll}
	(4,0) & \mbox{if} ~x_1=x_2=3, \\[3pt]
	(x_1,x_2) & \mbox{otherwise}.
	\end{array}\right.
	$$
	Consider the functions $\zeta(t,s)=\frac{4}{5}s-t$, $G(s,t)=s-t$ and $C_G=0$, then we note that
	$$0\leq \zeta\Big(d(fx,fy), M(x,y)\Big)$$
	holds if 
	$(x,y) \notin\left\lbrace\Big((4,5),(5,4)\Big), \Big((5,4),(4,5)\Big)\right\rbrace$.
	Since in this case
	$$\frac{1}{2}d(gx,fx)\geq d(gx,gy),$$
	so $f$ is a Suzuki type generalized $(Z_G, g)$-contraction.  Also, $f$ and $g$ satisfy all hypothesis of Theorem \ref{TheoremB} with the unique common fixed point $(x_1,x_2)=(0,0)$. Notice that
	$$\zeta\Big(d(fx,fy),M(x,y)\Big)<0$$
	whenever $(x,y) \in \left\lbrace\Big((4,5),(5,4)\Big), \Big((5,4),(4,5)\Big)\right\rbrace$. So the mapping $f$ does not satisfy the condition \eqref{RadnovicContrA} and \eqref{GeneralizedZG-contraction} and hence Theorem 2.1 and 2.3 of Radenovic and Chandok \cite{RadenovicChandok} are not applicable.
\end{example}
\section{Coincidence points of rational type quasi-contractions}
\label{4}
If $f$ and $g$ be two self-mappings on a metric space $(X,d)$ satisfying $f(X)\subseteq g(X)$ and $x_0 \in X$, let us define $x_1 \in X$ such that $fx_0 =gx_1$. Having defined $x_n\in X$, let $x_{n+1} \in X$ be such that $fx_n=gx_{n+1}$. We say that $\{fx_n\}$ is a $f$-$g$-sequence of initial point $x_0$. Letting $fx_n=gx_{n+1}=y_n$, where $n\in \mathbb{N}\cup\{0\}$, denote
$\mathcal{O}(y_k;n)=\{y_k, y_{k+1}, y_{k+2}, \cdots,y_{k+n}\}$. Let $\delta[\mathcal{O}(y_k;n)]$ denotes the diameter of $\mathcal{O}(y_k;n)$. If $\delta[\mathcal{O}(y_k;n)]>0$ for $k,n \in \mathbb{N}$ then we have $\delta[\mathcal{O}(y_k;n)]= d(y_i,y_j)$, where $k\leq i<j\leq k+n$.

\begin{definition}\rm \label{modifiedCGproperty}
	A mapping $G : [0,\infty)^2 \to \mathbb{R}$ has a property $C_G^m$, if there exists an
	$C_G \ge 0$ such that
	\begin{enumerate}
		\item[(i)]  $G(s, t) > C_G$ implies $\dfrac{s}{s+1} > t$,
		\item[(ii)] $G(t, t) \le C_G$ for all $t \in [0,\infty)$.
	\end{enumerate}
\end{definition}

\noindent Denote $m(x,y)=\max\left\lbrace d(gx,gy), d(gx,fx), d(gy,fy), d(gx,fy), d(gy,fx) \right\rbrace.$ 
Using the above property we prove following result:

\begin{theorem}\rm \label{TheoremThree3}
	Let $(X, d)$ be a metric space and $f, g : X \to X$ be self-mappings  satisfying the condition
	\begin{equation} \label{SuzukiUsingZetaExample}
	C_G \leq \frac{m(x,y)}{m(x,y)+1}-d(fx,fy) < G(m(x,y), d(fx,fy))
	\end{equation} 
	where $G : [0,\infty)^2 \to \mathbb{R}$  is a $C$-class function. 
	Assume that the following conditions hold:
	\begin{enumerate}
		\item[(i)] $f(X)\subseteq g(X)$,
		\item[(ii)] $g(X)$ or $f(X)$ is complete.
		\item[(iii)] if $\{t_n\}$ and $\{s_n\}$ are sequences in $(0,\infty)$ such that $\lim\limits_{n\to\infty} t_n = \lim\limits_{n\to\infty} s_n > 0$ and
		$t_n < \dfrac{s_n}{s_n +1}$, then $\lim\limits_{n\to\infty}\sup\left(\dfrac{s_n}{s_n +1}-t_n\right) <C_G$.
	\end{enumerate}
	Then $f$ and $g$ have unique point of coincidence.
\end{theorem}

\begin{proof}
	Suppose that the point of coincidence exists, then the uniqueness follows as: let $u_1$ and $u_2$ be two distinct points of coincidence of $f$ and $g$, i.e. $u_1=fv_1=gv_1$  and $u_2=fv_2=gv_2$ for some  $v_1,v_2 \in X$. 
	Then using \eqref{SuzukiUsingZetaExample} and (ii) of Definition \ref{modifiedCGproperty}, we get 
	\begin{align*}
	C_G & \leq \dfrac{m(v_1,v_2)}{m(v_1,v_2)+1}-d(fv_1,fv_2)<G\Big(m(v_1,v_2), d(fv_1,fv_2)\Big)\\[5pt]
	&\Longrightarrow C_G<G(d(u_1,u_2),d(u_1,u_2))\le C_G,
	\end{align*}
	which is a contradiction. Thus, the point of coincidence is unique.
	
	Now, consider the sequences $\{x_n\}$ and $\{y_n\}$ defined as
	$$y_n=fx_n=gx_{n+1}~\text{for all}~n \in \mathbb{N}\cup\{0\},$$
	and assume that $y_n \ne y_{n+1}$ for all $n \in \mathbb{N}\cup\{0\}.$
	Taking $x = x_{n}, y = x_{n+1}$ in \eqref{SuzukiUsingZetaExample}, we have
	\begin{align*}
	C_G &\leq \dfrac{m(x_{n},x_{n+1})}{m(x_{n},x_{n+1})+1} -d(fx_{n}, fx_{n+1}) < G\Big(m(x_{n},x_{n+1}), d(fx_{n}, fx_{n+1})\Big). 
	\end{align*}
	Using (i) of Definition \ref{modifiedCGproperty}, we get
	\begin{equation*}
	d(fx_{n}, fx_{n+1}) < \dfrac{m(x_{n},x_{n+1})}{m(x_{n},x_{n+1})+1}
	\end{equation*}
	which gives 
	\begin{equation*}
	d(y_n, y_{n+1})< \dfrac{\max\left\lbrace d(y_{n-1}, y_n), d(y_{n},y_{n+1}), d(y_{n-1}, y_{n+1})\right\rbrace}{\max\left\lbrace d(y_{n-1}, y_n), d(y_{n},y_{n+1}), d(y_{n-1}, y_{n+1})\right\rbrace+1} 
	\end{equation*}
	\begin{equation*}
	\text{i.e.} ~~\delta[\mathcal{O}(y_n;1)]< \frac{\delta[\mathcal{O}(y_{n-1};2)]}{\delta[\mathcal{O}(y_{n-1};2)]+1}~~~
	\end{equation*}
	
	\begin{equation}\label{AA}
	\text{or} ~~\delta[\mathcal{O}(y_n;1)]< \lambda_1\delta[\mathcal{O}(y_{n-1};2)],~~~
	\end{equation}
	\text{where}~$\lambda_1=\dfrac{1}{\delta[\mathcal{O}(y_{n-1};2)]+1} <1$.
	Again using \eqref{SuzukiUsingZetaExample} and (i) of Definition \ref{modifiedCGproperty}, we get
	\begin{align*}
	\delta[&\mathcal{O}(y_{n-1};2)]=d(y_{r},y_p),~~\text{where}~n-1\leq r< p \leq n+1\\
	& <  \dfrac{\max\left\lbrace d(y_{r-1}, y_{p-1}), d(y_{r-1},y_{r}), d(y_{p-1}, y_{p}), d(y_{r-1}, y_{p}), d(y_{p-1}, y_{r})\right\rbrace}{\max\left\lbrace d(y_{r-1}, y_{p-1}), d(y_{r-1},y_{r}), d(y_{p-1}, y_{p}), d(y_{r-1}, y_{p}), d(y_{p-1}, y_{r})\right\rbrace+1} 
	\end{align*} 
	\begin{equation} \label{AA1}
	\text{i.e.}~~~\delta[\mathcal{O}(y_{n-1};2)] < \dfrac{\delta[\mathcal{O}(y_{n-2};3)]}{\delta[\mathcal{O}(y_{n-2};3)]+1} = \lambda_2\delta[\mathcal{O}(y_{n-2};3)].
	\end{equation}
	From \eqref{AA} and \eqref{AA1}, we have
	\begin{equation}\label{AA2}
	\delta[\mathcal{O}(y_n;1)]< \lambda_1\lambda_2\delta[\mathcal{O}(y_{n-2};3)].
	\end{equation}
	
	\noindent Continuing in same manner, we get
	\begin{equation}\label{BB}
	\delta[\mathcal{O}(y_n;1)]< \lambda_1\lambda_2\ldots \lambda_n\delta[\mathcal{O}(y_{0};n+1)],
	\end{equation} 
	where $\delta[\mathcal{O}(y_{0};n+1)] = d(y_r,y_q)$ for some positive $0\leq r<q\leq n+1$ and $\lambda_i=\dfrac{1}{\delta[\mathcal{O}(y_{n-i};i+1)]+1} <1$. Let $\lambda=\max\{\lambda_1, \lambda_2,\ldots, \lambda_n\}$, then we have
	\begin{equation}\label{BB1}
	\delta[\mathcal{O}(y_n;1)]< \lambda^n\delta[\mathcal{O}(y_{0};n+1)].
	\end{equation} 
	\noindent Now
	\begin{align*}
	\delta[\mathcal{O}(y_{0};n+1)] = d(y_r,y_q) &\leq d(y_r,y_{r+1}) + d(y_{r+1},y_q)\\
	& = d(y_r,y_{r+1}) + d(fx_{r+1},fx_q)\\
	& \leq d(y_r,y_{r+1}) + \delta[\mathcal{O}(y_1,n)]\\
	& \leq d(y_r,y_{r+1}) + \frac{\delta[\mathcal{O}(y_{0};n+1)]}{\delta[\mathcal{O}(y_{0};n+1)]+1}\\
	& = d(y_r,y_{r+1}) + \beta\delta[\mathcal{O}(y_{0};n+1)],
	\end{align*}
	where $\beta = \dfrac{1}{\delta[\mathcal{O}(y_{0};n+1)]+1}< 1$.
	Hence
	\begin{equation}\label{CC}
	\delta[\mathcal{O}(y_{0};n+1)] \leq \frac{1}{1-\beta}d(y_r,y_{r+1})
	\end{equation}
	Thus, from \eqref{BB1} and \eqref{CC}, we have
	
	\begin{equation} \label{DD}
	\delta[\mathcal{O}(y_n;1)]< \frac{\lambda^n}{1-\beta} d(y_0,y_1)
	\end{equation} 
	with $r=0$. Therefore, because of $\lambda <1$, we get
	\begin{equation}\label{DD1}
	d(y_n, y_{n+1}) \rightarrow 0~~ as ~~n \to \infty.
	\end{equation} 
	Now, we claim that the sequence $\{y_n\}$ is Cauchy. If not, then by Lemma \ref{LemmaA} we have
	$$\lim\limits_{k\to\infty} d(x_{m(k)},x_{n(k)})=\lim\limits_{k\to\infty} d(x_{m(k)+1},x_{n(k)+1})=\varepsilon,$$ 
	and consequently, 
	$$\lim\limits_{k\to\infty}m(x_{m(k)+1}, x_{n(k)+1})=\varepsilon.$$
	Also, from \eqref{SuzukiUsingZetaExample} and (i) of Definition \ref{modifiedCGproperty}, we have  $$d(y_{m(k)+1}, y_{n(k)+1})<\dfrac{m(x_{m(k)+1}, x_{n(k)+1})}{m(x_{m(k)+1}, x_{n(k)+1})+1}.$$ 
	So by assumption (iii), we get 
	\begin{equation}
	C_G \leq \lim\limits_{k\to\infty}\sup \Big(\dfrac{m(x_{m(k)+1}, x_{n(k)+1})}{m(x_{m(k)+1}, x_{n(k)+1})+1}-d(y_{m(k)+1}, y_{n(k)+1})\Big) < C_G,
	\end{equation}
	a contradiction. Hence $\{y_n\}$ is a Cauchy sequence.
	
	Now, suppose  $g(X)$ is complete subspace of $X$. Then there exists a point, say $z\in X$ such that $y_n \to gz$, i.e. $gx_n \to gz$ as $n\to \infty$. We will show that $z$ is  the coincidence point of $f$ and $g$. We  can suppose $y_n \ne fz,  gz$ for  $n\in\mathbb{N}\cup\{0\}$, otherwise we are done. 
	
	Now, using \eqref{SuzukiUsingZetaExample} for $x=x_n$ and $y=z$, we have   
	$$C_G \leq  \dfrac{m(x_n, z)}{m(x_n, z)+1}-d(fx_n, fz) < G\Big(m(x_n, z), d(fx_n, fz)\Big),$$ and then by (i) of Definition \ref{modifiedCGproperty}, we get $ d(fx_n, fz)<\dfrac{m(x_n, z)}{m(x_n, z)+1}$. If we assume $fz\neq gz$, then taking limit as $n\to \infty$, we get $d(gz,fz)<d(gz,fz)$, a contradiction. Hence $gz=fz$. In case $f(X)$ is a complete subspace of $X$, the sequence $\{y_n\}$ converges in $g(X)$ since $f(X) \subseteq g(X)$. In same manner we can show that $gz=fz$. 
\end{proof}

\begin{example}\rm
	Let $X= (-2,2)$ be a metric space endowed with metric $d(x,y)= |x-y|$. Let $f$ and $g$ be self-mappings on $X$ defined by
	$$f(x) = \left\{\begin{array}{lll}
	0 & \mbox{if} ~-2<x<-1 ~\text{and}~1<x<2, \\[3pt]
	\dfrac{1+x}{3} & \mbox{if} ~-1\leq x < 0, \\[5pt]
	\dfrac{1-x}{3} & \mbox{if} ~ 0\leq x \leq 1,
	\end{array}\right.$$
	$$g(x) = \left\{\begin{array}{lll}
	1 & \mbox{if} ~-2<x\leq-1, \\[3pt]
	x & \mbox{if} ~-1< x < 1, \\[5pt]
	-1 & \mbox{if} ~ 1\leq x <2.
	\end{array}\right.~~~~~~~~~~~~~~~~~~~~~~$$
	Here Jungck's theorem \cite{Jungck1996} is not applicable since the mapping $g$ is discontinuous. However, for $G(s,t)=s-t$ and $C_G=0$, $f$ and $g$ satisfy all hypothesis of Theorem \ref{TheoremThree3} with the unique coincidence point $x=\frac{1}{4}$. 
\end{example}
\begin{remark}\rm
Let $m(x,y)=s$ and $d(fx,fy)=t$ for $x,y \in X$ with the metric $d$. Define a function $\zeta : [0,\infty)^2 \to \mathbb{R}$ by $\zeta(t,s)=\dfrac{s}{s+1}-t$, then clearly $\zeta\in Z_G$. We also observed that the Theorem \ref{TheoremThree3} is an approach to provide an partial answer to the question posed by Radenovic and Chandok in \cite{RadenovicChandok}. 
\end{remark}

Let $m(x,y)=s$ and $d(fx,fy)=t$ for $x,y \in X$ with the metric $d$. Define a function $\zeta : [0,\infty)^2 \to \mathbb{R}$ by $\zeta(t,s)=\dfrac{s}{s+1}-t$, then clearly $\zeta\in Z_G$. We also observed that the Theorem \ref{TheoremThree3} is an approach to provide an partial answer to the question posed by Radenovic and Chandok in \cite{RadenovicChandok}. \medskip

We conclude this section with the following open question:\\
\textbf{Question:} \\Whether the Theorem \ref{TheoremThree3} holds when the condition \eqref{SuzukiUsingZetaExample} is replaced with
\begin{equation*} 
\frac{1}{2}d(gx,fx)<d(gx,gy)\Rightarrow C_G \leq \frac{m(x,y)}{m(x,y)+1}-d(fx,fy) < G(m(x,y), d(fx,fy))?
\end{equation*}

\end{document}